\documentclass[oneside ,10 pt ]{article}
\usepackage[b5 paper ]{geometry}
\usepackage {amsfonts , amsmath ,latexsym , amssymb}
\usepackage {amsthm}
\usepackage {mathrsfs , upref}
\usepackage { mathptmx }
\usepackage {dea}
\usepackage{verbatim}
\usepackage{dsfont}

\usepackage{enumerate}
 
\newtheorem{Thm}{Theorem}[section]

\theoremstyle{definition}

\newcommand{\ds}{\displaystyle}
\newcommand{\R}{\mathbb{R}}
\newcommand{\N}{\mathbb{N}}

\newcommand{\ts}{\thinspace}
\newcommand{\clb}{\color{blue}}

\newcommand{\norm}[1]{\left \lVert #1 \right \rVert}
\newcommand{\be}{\begin{equation}}
\newcommand{\ee}{\end{equation}}
\newcommand{\bes}{\begin{equation*}}
\newcommand{\ees}{\end{equation*}}
\newcommand{\mc}{\mathcal}

\newcommand{\ve}{\varepsilon}

\newcommand{\la}{\langle}
\newcommand{\ra}{\rangle}
\newcommand{\lt}{\left(}
\newcommand{\rt}{\right)}

\newcommand{\lam}{\lambda}

\newcommand{\ben}{\begin{enumerate}}
\newcommand{\een}{\end{enumerate}}

\newcommand{\bpf}{\begin{proof}}
\newcommand{\epf}{\end{proof}}

\newcommand{\bsi}{\begin{itemize}[label={{\clb\ding{51}}}]}
\newcommand{\ei}{\end{itemize}}
\newcommand{\bci}{\begin{itemize}[label={{\clb\ding{43}}}]}

\newcommand{\ora}{\overrightarrow}

\newcommand{\spn}{\text{span}}
\newcommand{\im}{\text{Im}}

\begin{document}
\title[Sturm-Liouville BVP's]{Weakly Nonlocal Boundary Value Problems with Application to Geology}
\author[D. Maroncelli and E. Collins]{ Daniel Maroncelli$^*$ and Emma Collins}
\address{Department of Mathematics, College of Charleston, Charleston, SC 29424-0001, U.S.A\\ \email{maroncellidm@cofc.edu}}
\address{Robinson Design Engineers, 10 Daniel St. Charleston, SC 29407, U.S.A\\ \email{ec@robinsondesignengineers.com}}
\CorrespondingAuthor{D. Maroncelli}
\begin{abstract} 
In many cases, groundwater flow in an unconfined aquifer can be simplified to a one-dimensional Sturm-Liouville model of the form:
\begin{equation*}
x''(t)+\lambda x(t)=h(t)+\ve f(x(t)),\hspace{.1in}t\in(0,\pi)
\end{equation*}
subject to non-local boundary conditions
\begin{equation*}
x(0)=h_1+\ve\eta_1(x)\text{ and } x(\pi)=h_2+\ve\eta_2(x).
\end{equation*}
In this paper, we study the existence of solutions to the above Sturm-Liouville problem  under the assumption that $\ve$ is a small parameter.  Our method will be analytical, utilizing the implicit function theorem and its generalizations.
\end{abstract}

\maketitle
**\textbf{ This work has been accepted for publication in the journal Differential Equations \& Applications.}

\section{Introduction}
\setlength{\parskip}{.15cm plus 0.05cm minus 0.02cm}
\allowdisplaybreaks

Many groundwater systems in the coastal regions of the southeastern United States are characterized as shallow, unconfined aquifers. Understanding the flow of groundwater in such systems is integral to flood hazard management and contaminant remediation.
It is well known that groundwater flow through a porous material can be modeled by the diffusion equation,
\be \label{E:G}\frac{\partial}{\partial x}\big{(}K_{xx}\frac{\partial h}{\partial x}\big{)} + \frac{\partial}{\partial y}\big{(}K_{yy}\frac{\partial h}{\partial y}\big{)}+ \frac{\partial}{\partial z}\big{(}K_{zz}\frac{\partial h}{\partial z}\big{)} = S_s\frac{\partial h}{\partial t} -W;\ee
see \cite{Mcdonald1988,rushton1979seepage} for a thorough treatment of these ideas.
 
In the above equation, $x,y,z$ are spatial variables, $t$ represents time, $S_s$ denotes specific storage of the material (a constant), $h$ is the hydraulic head as a function of $x,y,z,$ and $t$, $W$ is the recharge or discharge of water into or out of the system as a function of $x,y,z,t,$ and potentially $h$, and $K$ is hydraulic conductivity which may be a function of $x,y,z,$ and $t$. Hydraulic conductivity is often assumed to be constant in time. Moreover, if one stratigraphic layer is assumed to have the same hydraulic conductivity in all spatial directions (i.e. $K_{xx} = K_{yy} = K_{zz}$), then the system is said to be isotropic. 

If we make the common assumption that the groundwater system is isotropic and additionally that hydraulic conductivity is constant, then \eqref{E:G} simplifies to 
\be K\big{(}\frac{\partial^2h}{\partial x^2} + \frac{\partial^2h}{\partial y^2} + \frac{\partial^2h}{\partial z^2}\big{)} = S_s \frac{\partial h}{\partial t} -W\ee
\begin{center}or\end{center} 
\be\label{E:D} S_s \frac{\partial h}{\partial t}-K\Delta h= W,\ee
where $\Delta$ is the Laplacian operator and  $K$ is the hydraulic conductivity in all directions. In applications, boundary conditions representing various properties of the aquifer are assumed in addition to $\eqref{E:D}$.

It is very common in groundwater modeling to make the simplifying assumption that the groundwater recharge function $W$ is strictly a function of $x,y, z$, and possibly $t$. However, in many cases, groundwater flow has a complicated dependence on water table height which leads to nonlinear recharge/discharge and  nonlinear boundary conditions of the groundwater flow equation \eqref{E:D}. For example, it is well documented that evaporation of water from a shallow aquifer is heavily dependent on water table height, see \cite{raghunath2006hydrology,robinson2017hydrology}. In semi-arid and arid environments where evaporation can comprise up to 70\% of the water budget {\cite{narasimhan2008note}}, these nonlinear conditions cannot be ignored.  It is this nonlinear dependence that is the focus of this paper.

The existing literature on nonlinear boundary value problems is extensive.  For ideas closely related to the ideas of this paper, we suggest \cite{Bouch,Chen,Chow,DrabekRes,DrabekLL,DrabekPLap,Du,MaSturm,MaroncelliSturm,MaroncelliNonlinearImp,Maroncelli2013,Maroncelli2014,RodSuaNon,RodSuaSturm}.

\section{Derivation of One-Dimensional Flow}
In situations of unidirectional groundwater flow with no seepage, flow can be described by the one dimensional flow equation 
\be\label{E:1Dflow}\frac{\partial u}{\partial t}  - \alpha \frac{\partial^2}{\partial x^2}u = \beta, \ee
see \cite{Kim} for the details,
with the constant $\alpha$ and the function $\beta$ determined by various aquifer properties and $u(t,x)$ representing water table height as a function of time $t$ and the spacial variable $x$. If we make the simplifying assumption that $u$ can be separated into a steady state and transient component, then we have
\be u(t, x) = v(x) + w(t, x)\ee
where $v$ is the steady state  component of $u$ and $w$ is the transient component of $u$. The steady state component can be thought of as persisting behaviors in the aquifer system such as flow due to aquifer slope, evaporation, etc., whereas the transient component can be thought of as groundwater behavior due to temporary influences such as flood, precipitation events, etc. The flow equation, \eqref{E:1Dflow},  now becomes
\bes
u_t - \alpha u_{xx} =w_t - \alpha v_{xx} - \alpha w_{xx} =\beta.
\ees
If we consider this simplified version of the model on the infinite strip $[0,\infty)\times [0,L]$, and impose nonlinear boundary and initial conditions, we arrive at
\bes 
\begin{split}
w_t - \alpha v_{xx} - \alpha w_{xx} &= \beta\\
u(t, 0) = v(0) + w(0, t) &= h_1(t)+\gamma_1\\
u(t, L) = v(L)+w(L, t) &= h_2(t)+\gamma_2\\
u(0, x) = v(x)+w(x, 0)&= h_0(x)
\end{split},
\ees
where we assume $h_1, h_2$ and $h_0$ are arbitrary real-valued continuous functions with $h_i:=\lim_{t\to\infty}h_i(t)$ existing for $i=1,2$ and $\gamma_1, \gamma_2$ are possibly nonlinear functionals on the space of continuous functions, $C([0,\infty)\times [0,L])$.
Note that constant-head boundary conditions, in which $h_1(\cdot)$ and $h_2(\cdot)$ are constants  and $\gamma_1=0=\gamma_2$, correspond to Dirichlet boundary conditions. Other commonly modeled boundary conditions are discussed in \cite{franke1987definition}. 

As described in \cite{Kim}, it is reasonable to assume the transient component $w$ and its derivatives vanish as $t\rightarrow \infty$ (as the system reaches equilibrium). Under this assumption, and appropriate convergence assumptions, we deduce, as $t\rightarrow \infty$, 
\bes-\alpha v_{xx} = \beta\ees
and
\bes
v(0)=h_1+\gamma_1 \text{ and } v(L)=h_2+\gamma_2.
\ees
Since we are considering the one-dimensional case, this equates to
\be \label{E:1D_sturm} v_{xx} =  v''(x) = -\beta/\alpha \ee
with boundary conditions 
\be\label{E:1D_bound} 
v(0)=h_1+\gamma_1 \text{ and } v(L)=h_2+\gamma_2.
\ee

In general, $\beta$ is a function of water table height, $u$, but a common practice in groundwater modeling is to assume $\alpha$ and $\beta$ are constant and that $\gamma_1=0=\gamma_2$.  This is often referred to as ``constant-head'' boundary conditions. Under this assumption we are lead to the well known result in hydrology (see  \cite{Fetter}),
\bes
v(x)=\frac{-\beta}{2\alpha}x^2+(\frac{h_2-h_1}{L}+\frac{\beta}{2\alpha}L)x+h_1.
\ees
Therefore, solving the groundwater flow equation \eqref{E:1Dflow} is reduced to solving the homogeneous linear PDE describing the transient component
\be w_t = \alpha w_{xx}\ee
which can be easily solved with well-known results from PDEs upon defining an initial condition, $h_0$. 
When $\beta$ is not assumed to be constant and/or $\gamma_1, \gamma_2$ are assumed to be nonzero, the situation is much more complicated. It is in this direction that we turn for the for the remainder of the paper.

  \section{Weak Nonlinearities}

\subsection{Preliminaries}
Based on our derivations in previous section, we will now study one-dimensional groundwater flow in an unconfined aquifer via \eqref{E:1D_sturm}-\eqref{E:1D_bound}.  For many aquifer systems under natural conditions, it is reasonable to assume that $\beta$, representing groundwater recharge, has an ``almost''  linear relationship to groundwater height; that is, $\beta$ consists of a linear term and some ``small''  possibly nonlinear term. From \eqref{E:1D_sturm}, consider the case where  $\beta(x, v) = \lambda v-h-\epsilon f(x,v)$ for $\lambda, \ve\in \R$ and some fixed continuous functions $h:\R\to \R$ and $f:\R^2\to \R$.  In addition, assume that for $i=1,2$, $\gamma_i=\ve \eta_i$  where $\eta_1, \eta_2$ are smooth functionals. A concrete example of $\eta_1, \eta_2$ might be multi-point boundary operators such as
\bes
\eta_1(v)=\sum_{k=1}^n g_k(v(t_k)),
\ees
\bes
\eta_2(v)=\sum_{j=1}^m h_j(v(t_j)),
\ees
where each $g_k$, $h_j$ is a differentiable function and each $t_k$, $t_j\in [0,L]$.

 If we assume, without loss of generality, that $\alpha= 1$ and $L=\pi$, then  \eqref{E:1D_sturm}-\eqref{E:1D_bound} becomes
\be \label{Eq:sturm}
v''(x)+\lambda v(x)= h(x)+ \epsilon f(x,v(x)), \ts x\in (0,\pi),\ee 
subject to
\be\label{E:bc_1} v(0) = h_1+\epsilon\eta_1(v) \ee
\be \label{E:bc_2} v(\pi) =h_2+ \epsilon\eta_2(v).\ee
Equation \eqref{Eq:sturm} along with boundary conditions \eqref{E:bc_1} and \eqref{E:bc_2} is a special case of  what is often referred to as a regular Sturm-Liouville boundary value problem. 

We choose to study \eqref{Eq:sturm}-\eqref{E:bc_2} through the use of operators on Banach spaces.  Before getting to our main result, we introduce appropriate spaces and operators. We let $C:=C[0,\pi]$ denote the space of real-valued continuous functions topologized by the supremum norm, $\norm{\cdot}_\infty$. The space $L^2:=L^2[0,\pi]$ will denote, as usual, the space of real-valued square-integrable functions defined on $[0,\pi]$. The topology on $L^2$ will be that induced by the standard $L^2$-norm, $\norm{\cdot}_2$. We use  $H^2$ to denote the Sobelov space of functions with two weak derivatives in $L^2$;  that is,
\bes H^2=\{x\in L^2\mid x' \text{ is absolutely continuous and } x''\in L^2\}.\ees  Unless otherwise stated, the topology on $H^2$ will be the subspace topology inherited from $L^2$.
On occasion, we may also view $H^2$ as a subspace of $C$. Each scenario should be clear from the context of our discussion. Finally, we will use $|\cdot|$ to denote the Euclidean norm on $\R^2$, and $\langle\cdot, \cdot\rangle_2$ and $\langle\cdot, \cdot\rangle_\R$ will denote the inner products on  $L^2$ and $\R^2$, respectively.  

 For each $\lam\in \R$, we define a differential operator $\mc{A}_\lam:H^2\rightarrow L^2 $ by
\bes
\mc{A}_\lam v = v'' +\lambda v.
\ees
 For $i=1,2$, we define  boundary operators $B_i:H^2\to \R$ by
\bes B_1:H^2\rightarrow \R \text{ by }B_1v= v(0),\ees
and
\bes B_2:H^2\rightarrow \R\text{ by }B_2v= v(\pi).\ees
Finally, we let 
\begin{align*}
&\mc{L}_\lam:H^2\rightarrow L^2\times \R \times \R\text{ by }
\mc{L}v = \begin{bmatrix} 
\mc{A}_\lam v \\
B_1v\\
B_2v
\end{bmatrix}, \text{ and}\\
&\mc{F}:H^2\rightarrow L^2\times \R\times \R\text{ by }
\mc{F}(v) = \begin{bmatrix} 
f(\cdot, v) \\
\eta_1(v)\\
\eta_2(v) 
\end{bmatrix}.
\end{align*}
With this notation, solving \eqref{Eq:sturm} with boundary conditions \eqref{E:bc_1} and \eqref{E:bc_2}  is equivalent to solving $\mc{L}_\lam v = \ora{h}+\epsilon\mc{F}(v)$, where 
\bes\ora{h} =  \begin{bmatrix} 
h\\
h_1\\
h_2
\end{bmatrix}.
\ees

The study of the nonlinear boundary value problem \eqref{Eq:sturm}-\eqref{E:bc_2} will be intimately related to the linear nonhomogeneous boundary value problem 
\begin{equation}
v''(x)+\lambda v(x)=h(x),\label{E:NHSturm} \hspace{.2in} x\in (0,\pi)
\end{equation}
\begin{equation}
v(0)= h_1\text{ and }v(\pi)=h_2,\label{E:NHSBoundary}
\end{equation}
where $h$ is an arbitrary element of $L^2$ and $h_1$ and $h_2$ are elements of $\R$. 
Using our notation from above, we have that solving \eqref{E:NHSturm}-\eqref{E:NHSBoundary} is equivalent to solving 
\be\label{E:NonHom} \mc{L}_\lam v=\ora{h}=\begin{bmatrix}h\\h_1\\h_2\end{bmatrix}.\ee

The analysis of the boundary value problem \eqref{Eq:sturm}-\eqref{E:bc_2} follows two distinct routes, one in which $\lam$ is an eigenvalue of the operator $v\to -v''$ (subject to Dirichlet boundary conditions $v(0)=0=v(\pi)$) and one in which it is not.   The difficulty in the analysis lies mostly in the case where $\lam$ is an eigenvalue; this is, $\lam=n^2$ for some natural number $n$.  When $\lam$ is an eigenvalue, this case is often referred to as the case of resonance. It is the case of resonance that we are most interested in, however, for completeness, we we include, in what follows, an analysis of \eqref{Eq:sturm}-\eqref{E:bc_2} for the case  in which $\lam$ is not an eigenvalue.

\subsection{Invertible $\mc{L}_\lam$}

It is well-known that when $\lam\neq n^2$ for some natural number $n$, that is, when $\lam$ is not an eigenvalue of the mapping  $v\to -v''$ (subject to Dirichlet boundary conditions $v(0)=0=v(\pi)$), $\mc{L}_\lam$ is an invertible operator.  Under this assumption, the existence of solutions to \eqref{Eq:sturm}-\eqref{E:bc_2} follows easily under very mild assumptions on $f$ (see \eqref{Eq:sturm} for notation). We present two such cases.


\begin{Thm}
Suppose that $f$ is continuously differentiable with respect to its second component, that for each $i=1, 2$, $\eta_i$ is continuously differentiable (relative to the supremum norm), and $\lam\not\in\N^2$. Then there exists a $\delta> 0$ such that if $-\delta<\ve <\delta$, the boundary value problem \eqref{Eq:sturm}-\eqref{E:bc_2} has a unique solution $v_\ve$.  Moreover, the mapping $\ve\to v_\ve$ is continuously differentiable.

\end{Thm}

\begin{proof}
Our proof is an application of the implicit function theorem.   Under the assumption that $f$ is continuously differentiable in its second component, it is an easy exercise to show that $\mc{F}$ is continuously differentiable on $C[0,\pi]$ (when given the supremum norm).  The fact that $\lam\not\in\N^2$ allows us to define a function $G_\lam:\R\times C[0,\pi]\to C[0,\pi]$ by 
\bes
G_\lam(\ve, v)=v-\mc{L}_\lam^{-1}\ora{h}-\ve\mc{L}_\lam^{-1} \mc{F}(v).
\ees
Note that here we are using $\mc{L}_\lam$ to denote the restriction of $\mc{L}_\lam$ to $C^2[0,\pi]$. 

We take a moment to point out a few obvious facts. First, it is clear that $v\in C[0,\pi]$ solves \eqref{Eq:sturm}-\eqref{E:bc_2} if and only if $G_\lam(\ve, v)=0$.  Secondly, $G_\lam$ is continuously differentiable since $\mc{F}$ is.  Finally, $G_\lam(0, \mc{L}_\lam^{-1}\ora{h})=0$.

 Now, differentiating $G_\lam$ with respect to $v$, we see that 
\bes
\dfrac{\partial G_\lam}{\partial v}(0, \mc{L}_\lam^{-1}\ora{h})=I,
\ees
which is certainly a topological isomorphism. Under these conditions, by the implicit function theorem, there must exists a $\delta>0$ and a unique $C^1$ function $u:(-\delta, \delta)\subset\R \rightarrow C[0,\pi]$ such that $G(\epsilon, u(\epsilon)) = 0$ for each $\ve\in (-\delta, \delta)$. The proof is now complete.
\end{proof}

We finish the case of invertible $\mc{L}_\lam$ under an assumption of ``Lipschitzness''.
\begin{Thm}\label{Thm:Invertible_Lipschitz}
Suppose $f$ is Lipschitz with respect to its second component, that for each $i=1, 2$, $\eta_i$ is Lipschitz, and that $\lam\not\in\N^2$. Then there exists a $\delta> 0$ such that if $-\delta<\epsilon<\delta$, the boundary value problem \eqref{Eq:sturm}-\eqref{E:bc_2} has a unique solution. \end{Thm}
\begin{proof}
For each $\ve\in \R$, define  $H_\lam:\R\times C[0,\pi]\to C[0,\pi]$ by 
\bes
H_\lam(\ve, v)=\mc{L}_\lam^{-1}\ora{h}+\ve \mc{L}_\lam^{-1}(\mc{F}(v)).
\ees 
For each fixed $\ve$, the solutions of \eqref{Eq:sturm}-\eqref{E:bc_2} are the fixed points of $H_\lam(\ve, \cdot)$. 
By assumption, $\mc{F}$ is Lipschitz, say with Lipschitz constant $\alpha\geq 0$. Note that since $\mc{L}_\lam$ is a differential operator, $\mc{L}_\lam^{-1}$ will be an integral operator, and is therefore linear and bounded. 
Thus, for a fixed $\ve$ and  and any $u, w\in C[a, b]$, 
\begin{align*}
    \norm{H_\lam(\epsilon, u) - H_\lam(\epsilon, w)}_\infty &= |\epsilon| \norm{\mc{L}_{\lam}^{-1}\mc{F}(u) - \mc{L}_{\lam}^{-1}\mc{F}(w)}_\infty\\
    &= |\epsilon|\norm{\mc{L}_{\lam}^{-1}(\mc{F}(u) - \mc{F}(w))}_\infty\\
    &\leq |\epsilon|\norm{\mc{L}_{\lam}^{-1}}\norm{\mc{F}(u) - \mc{F}(w)}\\
    &\leq |\epsilon|\norm{\mc{L}_{\lam}^{-1}}\cdot\alpha\norm{u-w}_\infty.
\end{align*}
If $0\leq|\epsilon|\norm{\mc{L}_{\lam}^{-1}}\alpha <1$, then $H_\lam(\ve, \cdot)$ is a contraction and from the contraction mapping theorem has exactly one fixed point. 
\end{proof}

\subsection{Noninvertible $\mc{L}_\lam$}

We now come to the main focus of this paper, that is, existence results for problem \eqref{Eq:sturm}-\eqref{E:bc_2} under the assumption that $\lam=n^2$. It is well known that under this assumption  $\mc{L}_{n^2}$ is not invertible and that $\ker(\mc{L}_{n^2}) =\spn\{\sin(nt)\}$. Define $\psi_{n}(t) = \sqrt{\frac{2}{\pi}}\sin(nt)$ so that $\psi_{n}$ is an $L^2$-normalized basis for $\ker(\mc{L}_{n^2})$.
Further, define $\varphi_{n}(t) = \dfrac{-\sqrt{\frac{\pi}{2}}\cos(nt)}{n}$. A simple calculation shows that $\la\psi_{n}, \varphi_{n}\ra _2 = 0$, where again $\la\cdot, \cdot\ra _2$ denotes the standard inner product on $L^2=L^2[0,\pi]$.  Moreover, we have that 

\begin{align*}
wr(\psi_{n}, \varphi_{n}) &= \begin{vmatrix} 
\sqrt{\frac{2}{\pi}}\sin(nt) &  \dfrac{-\sqrt{\frac{\pi}{2}}\cos(nt)}{n} \\
\sqrt{\frac{2}{\pi}} n\cos(nt) & \sqrt{\frac{\pi}{2}}\sin(nt)\\
\end{vmatrix}\\
&= 1.
\end{align*}
It follows from the theory of second-order linear differential equations that $\{\psi_{n},\varphi_{n}\}$ is a basis for the $\ker(\mc{A}_{n^2})$.

In cases where $\mc{L}_\lam$ is not invertible, the methods of the previous section clearly fail, since they depend crucially on the existence of $\mc{L}_\lam^{-1}$. Thus, proving the existence of solutions in the case when $\lam=n^2$ for some natural number $n$ will have to proceed via an alternate route.  Our approach in this direction will be to introduce a projection scheme (Lyapunov-Schmidt) which will allow us to reduce the problem of solving \eqref{Eq:sturm}-\eqref{E:bc_2} to solving a system of equivalent equations. The construction of our projection scheme will depend heavily on a  characterization of the image of $\mc{L}_\lam$, thus, we begin by characterizing $\im(\mc{L}_{n^2})$ for any natural number $n$. Our characterization is in the spirit of results found in \cite{Maroncelli2014}. 

We start by defining 
\[ \omega_n(t,s) = \begin{cases} 
      \psi_{n}(t)\varphi_{n}(s)& 0\leq t\leq s \leq 1 \\
      \psi_{n}(s)\varphi_{n}(t) & 0\leq s\leq t\leq 1 \\
   \end{cases},
\]
that is,
\[ \omega_n(t,s) = \begin{cases} 
     \dfrac{-\sin(nt)\cos(n s)}{n} &0\leq t\leq s \leq 1 \\\\
     \dfrac{-\sin(n  s)\cos(n   t)}{n } & 0\leq s\leq t\leq 1 \\
   \end{cases}.
\]
Now define $K_n:L^2\to H^2$ by 
\bes\label{E:Green}
K_nh(t)=\int_0^1\omega_n(t,s)h(s)ds.
\ees
From standard results in analysis, $K_n$ is compact, self-adjoint, and satisfies $\mc{A}_{n^2}K_nh=h$ for every $h\in L^2$.  By direct calculation, one easily establishes that for every $h\in L^2$, $B_1K_nh=\la h,\varphi_{n}\ra_2B_1\psi_{n}=0$ and  $B_2K_nh=\langle h,\psi_{n}\rangle_2B_2\varphi_{n}$.
Let 
\bes
v_{1,n}=B_1\varphi_{n}=-\sqrt{\dfrac{\pi}{2n^2}} \text{ and } v_{2,n}=B_2 \varphi_{n}=(-1)^{n+1}\sqrt{\dfrac{\pi}{2n^2}}.
\ees

We are now in a  position to characterize $\im(\mc{L}_{n^2})$.
To this end, define an inner product on $L^2[0,\pi]\times \R^2$ by 
\be\label{E:inner_product}\la \begin{bmatrix} 
h \\
w_1\\
w_2
\end{bmatrix},  \begin{bmatrix} 
g\\
y_1\\
y_2
\end{bmatrix}\ra  = \frac{\pi}{\pi+4n^2}\bigg{(} \la h,g\ra _2 + \la \begin{bmatrix} 
w_1\\
w_2
\end{bmatrix},  \begin{bmatrix} 
y_1 \\
y_2
\end{bmatrix}\ra _{\R}\bigg{)}.
\ee
Further, define
\bes
\ora{\psi_{n}} = \begin{bmatrix} 
\psi_n \\
v_{1,n}^{-1}\\
-v_{2,n}^{-1}
\end{bmatrix}
\ees 
and for $h\in L^2$, $h_1, h_2 \in \R$, define
\bes
\ora{h} =  \begin{bmatrix} 
h\\
h_1\\
h_2
\end{bmatrix}.
\ees
Note that $\norm{\ora{\psi_{n}}} = 1$ for the norm generated by the inner product on $L^2\times \R^2$. It can be shown, see  \cite{Maroncelli2014}, that $\ora{h}$ is in the image of $\mc{L}_{n^2}$ iff $\la\ora{h}, \ora{\psi_{n}}\ra  = 0$. Since

\begin{align*}
\la\ora{h}, \ora{\psi_{n}}\ra  &= \frac{\pi}{\pi+4n^2}\bigg{(}\la\sqrt{\frac{2}{\pi}}\sin(nt), h(t)\ra _2 + \la \begin{bmatrix} 
-\sqrt{\frac{2n^2}{\pi}}
\\
(-1)^n\sqrt{\frac{2n^2}{\pi}}
\end{bmatrix},  \begin{bmatrix} 
h_1 \\
h_2
\end{bmatrix}\ra _{\R} \bigg{)} \\
&=\frac{\pi}{\pi+4n^2}\bigg{(}\int_0^{\pi} \sqrt{\frac{2}{\pi}}\sin(nt)h(t)dt+\sqrt{\frac{2n^2}{\pi}}( (-1)^nh_2-h_1)\bigg{)},\\
\end{align*}
 we see that $\ora{h}$ is in the image of $\mc{L}_{n^2}$ if and only if
 \be\label{E:ImLCh}
 \ds n(h_1+(-1)^{n+1}h_2) = \int_0^{\pi} \sin(nt)h(t)dt.
 \ee  Note that if $h_1=0=h_2$, we get the very famous result that $h\in L^2$ is in the image of $\mc{A}_{n^2}$ (with Dirichlet boundary condtions) if and only if $\ds\int_0^\pi \sin(nt)h(t)dt=0$.

We now introduce projection operators for our Lyapunov-Schmidt projection scheme.  Since $\ker(\mc{L}_{n^2})$ is spanned by $\psi_{n}$, the orthogonal projection operator $P_n:L^2\to L^2$ onto the kernel of $\mc{L}_{n^2}$ is given by 
\bes
P_n(x) = \la x, \psi_{n}\ra _2\psi_{n}.
\ees
Also, we define $Q_n:L^2\times \R^2\to L^2\times \R^2$ by \bes Q_n(\ora{h}) = \la\ora{h}, \ora{\psi_{n}}\ra \ora{\psi_{n}}.\ees
From our analysis above, $Q_n$ is a projection onto $\im(\mc{L}_{n^2})^{\bot}$, and so $I-Q_n$ is a projection onto $\im(\mc{L}_{n^2})$. 

 We will proceed by applying the Lyapunov-Schmidt projection scheme, as outlined in \cite{Maroncelli2013, Maroncelli2014}. Since  $I- Q_n$ is the orthogonal projection onto the image of $\mc{L}_{n^2}$ and $Q_n$ is a projection onto $\im(\mc{L}_{n^2})^{\bot}$, we get 
\bes \mc{L}_{n^2}v = \ora{h} +\epsilon \mc{F}(v)\ees
if and only if 
\bes\left\{ \begin{array}{c} 
    (I-Q_n)\mc{L}_{n^2}v= (I-Q_n)\ora{h} + (I-Q_n)\epsilon \mc{F}(v)\\
    \text{and}\\
     Q_n\mc{L}_{n^2}v = Q_n\ora{h} + Q_n\epsilon \mc{F}(v)\\
   \end{array}\right..
\ees
Since $\mc{L}_{n^2}u \in \im(\mc{L}_{n^2})$, $(I-Q_n)\mc{L}_{n^2}u = \mc{L}_{n^2}u$. Likewise, for $\ora{h} \in \im(\mc{L}_{n^2})$, $(I-Q_n)\ora{h} = \ora{h}$, and so solving \eqref{Eq:sturm}-\eqref{E:bc_2} is equivalent to solving 
\bes
\left\{\begin{array}{c}
    \mc{L}_{n^2}v = \ora{h} + \epsilon (I-Q_n)\mc{F}(v)\\
    \text{and}\\
     0 =  \epsilon Q_n\mc{F}(v)\\
   \end{array}\right..
\ees
Let $\mc{M}_n$ denote the generalized inverse of $\mc{L}_{n^2}$ such that $\mc{M}_{n}:\im(\mc{L}_{n^2})\rightarrow \ker(\mc{L}_{n^2})^{\bot}$. By applying $\mc{M}_n$ to the first equation of our system, we get
\[ \Leftrightarrow \begin{cases} 
   \mc{M}_n \mc{L}_{n^2}v &= \mc{M}_n\ora{h} + \epsilon \mc{M}_n(I-Q_n)\mc{F}(v)\\
       \epsilon Q_n\mc{F}(v) &= 0\\
   \end{cases}.
\]
Recalling that $P_n$ was the projection onto $\ker(\mc{L}_{n^2})$, we see $\mc{M}_n\mc{L}_{n^2} = I-P_n$. Thus, our system is equivalent to
\[\begin{cases} 
   (I-P_n)v &= \mc{M}_n\ora{h} + \epsilon \mc{M}_n(I-Q_n)\mc{F}(v)\\
     \epsilon Q_n\mc{F}(v) &= 0\\
   \end{cases}.
\]
We are interested in finding solutions to \eqref{Eq:sturm}-\eqref{E:bc_2} for $\epsilon \neq 0$, so we will assume that \bes \epsilon Q_n\mc{F}(u) = 0 \Leftrightarrow Q_n\mc{F}(u) = 0.\ees
Thus, the system becomes
\[ \Leftrightarrow \begin{cases} 
   (I-P_n)u &= \mc{M}_n\ora{h} + \epsilon \mc{M}_n(I-Q_n)\mc{F}(u)\\
     Q_n\mc{F}(u) &= 0\\
   \end{cases}.
\]
Since we are assuming $h\in C[0,\pi]$, we can define $G_n:\R\times C[0,\pi]\to\im(I-P_n)\times \im(Q_n)$ by 
\[ G_n(\epsilon, u) =  \begin{cases} 
   (I-P_n)u - \mc{M}_n\ora{h} - \epsilon \mc{M}(I-Q_n)\mc{F}(u)\\
     Q_n\mc{F}(u) \\
   \end{cases}.
\]
Solving \eqref{Eq:sturm}-\eqref{E:bc_2} is now equivalent to 
\bes
G_n(\ve, u)=  \begin{bmatrix} 
   (I-P_n)u - \mc{M}_n\ora{h} - \epsilon \mc{M}_n(I-Q_n)\mc{F}(u)\\
     Q_n\mc{F}(u) \\
   \end{bmatrix}=0.
\ees

Our goal is to apply the implicit function theorem to prove the existence of solutions to \eqref{Eq:sturm}-\eqref{E:bc_2}.  For reference in what is to follow, we point out here that $G_n$ is continuously differentiable (relative to the supremum norm) with\
\bes
D_2G(\ve,u)=\frac{\partial G_n}{\partial u}(\epsilon, u)  = \begin{bmatrix} 
   (I-P_n) - \epsilon \mc{M}_n(I-Q_n)D\mc{F}(u)\\
     Q_nD\mc{F}(u)\\
   \end{bmatrix},
   \ees
   where
 \be\label{Der}
 D\mc{F}(u)=\begin{bmatrix}
f_2(\cdot, u) \\
\eta_1'(u)\\
\eta_2'(u)\\
\end{bmatrix}.
\ee
For a proof along these lines, see \cite{Maroncelli2013}.

We now come to our main existence theorem in the case of noninvertible $\mc{L}_\lam$.  We prove the existence of solutions to problem \eqref{Eq:sturm}-\eqref{E:bc_2} under mild assumptions on the solutions to the nonhomogeneous problem \eqref{E:NHSturm}-\eqref{E:NHSBoundary}.
\begin{Thm}\label{T:mainresonance} Suppose $\lam=n^2$ for some $n\in \N$ and that $\bar{u}$ is a particular solution to the nonhomogeneous equation \eqref{E:NHSturm}-\eqref{E:NHSBoundary}.
If \be \label{E:Imcond}n\bigg{[}\eta_1(\bar{u})+(-1)^{n+1}\eta_2(\bar{u})\bigg{]} = \int_0^{\pi} f(s,\bar{u}(s))\sin(ns)ds\ee and
\be\label{E:Dercond} n\bigg{[}\eta_1'(\bar{u})(\sin(n\cdot))+(-1)^{n+1}\eta_2'(\bar{u})(\sin(n\cdot))\bigg{]} \neq \int_0^{\pi} f_2(s,\bar{u}(s))\sin^2(n s)ds,\ee
then \eqref{Eq:sturm} subject to boundary conditions  \eqref{E:bc_1} and \eqref{E:bc_2} has at least one solution for small $\epsilon$.
\end{Thm}
\begin{proof}
First, note that under the assumptions on $\bar{u}$, we have $G_n(0,\bar{u})=0$. If we can show that $D_2G_n$ is invertible as a continuous linear map, then the result will follow from the implicit function theorem. As a consequence of the open mapping theorem,  $D_2G_n$ is a topological isomorphism if and only if it is a bijection. We start by showing $D_2G_n$ is injective. 

Suppose for the moment that 

\bes
D_2G_n(0,\bar{u})(w)= \begin{bmatrix}
   (I-P_n)(w)\\
     Q_nD\mc{F}(\bar{u})(w)\\
   \end{bmatrix}=0.
\ees
 Since $w=P_n(w)$, we conclude that $Q_nD\mc{F}(\bar{u})(P_n(w))=0$.  Using \eqref{Der} we conclude
 \bes
\frac{2}{\pi}\la w,\sin(n\cdot)\ra _2Q_n\begin{bmatrix}
f_2(\cdot,\bar{u})\sin(n\cdot) \\
\eta_1'(\bar{u})(\sin(n\cdot))\\
\eta_2'(\bar{u})(\sin(n\cdot))\\
\end{bmatrix}\\\\
=0
\ees
or
\bes
\frac{2}{\pi}\la w,\sin(n\cdot)\ra_2\cdot\la\begin{bmatrix}
f_2(\cdot,\bar{u})\sin(n\cdot) \\
\eta_1'(\bar{u})(\sin(n\cdot))\\
\eta_2'(\bar{u})(\sin(n\cdot))\\
\end{bmatrix}, 
\begin{bmatrix}
\sqrt{\frac{2}{\pi}}\sin(n\cdot) \\
-\sqrt{\frac{2n^2}{\pi}}\\
(-1)^n\sqrt{\frac{2n^2}{\pi}}
\end{bmatrix}\ra\begin{bmatrix}
\sqrt{\frac{2}{\pi}}\sin(n\cdot) \\
-\sqrt{\frac{2n^2}{\pi}}\\
(-1)^n\sqrt{\frac{2n^2}{\pi}}
\end{bmatrix}=0.
\ees
Using the definition of $\la\cdot, \cdot\ra$, we deduce
\bes
\lt\frac{2}{\pi}\rt^{3/2}\cdot\la w,\sin(n\cdot)\ra_2\cdot\bigg{(}\int_0^{\pi} f_2(s,\bar{u}(s))\sin^2(n s)ds-n\eta_1'(\bar{u})(\sin(n\cdot))+(-1)^nn\eta_2'(\bar{u})(\sin(n\cdot))\bigg)=0.
\ees
\
Since we are assuming 
\bes n\bigg{[}\eta_1'(\bar{u})(\sin(n\cdot))+(-1)^{n+1}\eta_2'(\bar{u})(\sin(n\cdot))\bigg{]} \neq \int_0^{\pi} f_2(s,\bar{u}(s))\sin^2(n s)ds,\ees
we conclude that $\la w,\sin(n\cdot)\ra_2=0$, which is only the case when $w=0$ (since $P_n(w)=w$). It follows that $D_2G_n(0,\bar{u})$ is injective.

We finish the proof by showing that $D_2G_n(0,\bar{u})$ is onto. To this end, let $q$ be an arbitrary element in $\im(I-P_n)$, $p$ an element of $\im(Q_n)$ and define $c\in \R$ by
\bes c =  \frac{\la p,\ora{\psi_n}\ra - \la D\mc{F}(\bar{u})(q), \ora{\psi_{n}}\ra}{\la D\mc{F}(\bar{u})(\psi_{n}),  \ora{\psi_{n}}\ra },\ees
for $\bar{u}$ the  particular solution to the nonhomogeneous linear equation. Note that this is well defined since in the proof that $D_2G_n(0,\bar{u})$ is injective, we showed

\bes
\la D\mc{F}(\bar{u})(\psi_{n}),  \ora{\psi_{n}}\ra =\frac{2}{\pi}\bigg{(}\int_0^{\pi} f_2(s,\bar{u}(s))\sin^2(n s)ds-n\eta_1'(\bar{u})(\sin(n\cdot))+(-1)^nn\eta_2'(\bar{u})(\sin(n\cdot))\bigg),
\ees
which is nonzero by the assumption on $\bar{u}$.

Consider, $r = c\psi_{n} + q$. Then
\bes
D_2G_n(0, \bar{u})(r) = \begin{bmatrix} 
   (I-P_n)(r)\\
     Q_nD\mc{F}(\bar{u})(r)\\
   \end{bmatrix}.
\ees 
Since $\psi_{n}$ forms a basis for $\im(P_n)$, we deduce
\bes
D_2G_n(0, \bar{u})(r)  = \begin{bmatrix} 
  q\\
     Q_nD\mc{F}(\bar{u})(r)\
  \end{bmatrix}.
\ees
Recalling that 
\bes
    c= \frac{\la p,\ora{\psi_n}\ra  - \la D\mc{F}(\bar{u})(q), \ora{\psi_{n}}\ra }{\la D\mc{F}(\bar{u})(\psi_{n}),  \ora{\psi_{n}}\ra },
\ees
we get that 
\bes
\begin{split}
  Q_nD\mc{F}(\bar{u})(r)& = \frac{\la p,\ora{\psi_n}\ra  - \la D\mc{F}(\bar{u})q, \ora{\psi_{n}}\ra }{\la D\mc{F}(\bar{u})\psi_{n},  \ora{\psi_{n}}\ra }\cdot\la D\mc{F}(\bar{u})\psi_{n}, \ora{\psi_{n}}\ra +\la D\mc{F}(\bar{u})q, \ora{\psi_{n}}\ra )\ora{\psi_{n}}\\
    &=(\la p,\ora{\psi_n}\ra - \la D\mc{F}(\bar{u})q, \ora{\psi_{n}}\ra +\la D\mc{F}(\bar{u})q,  \ora{\psi_{n}}\ra )\ora{\psi_{n}}\\
    &= \la p,\ora{\psi_n}\ra \ora{\psi_{n}}\\
    &=p.
    \end{split}
\ees
It follows that $D_2G_n(0,\bar{u})$ is onto and therefore invertible. By the open mapping theorem, $D_2G_n$ has a continuous inverse and the result now follows from the implicit function theorem.
\end{proof}

\section{An Example}

In this section we give a concrete example of the application of our main result, Theorem \ref{T:mainresonance}.

Consider
 \begin{equation}\label{E:ODEEx}
v''+n^2v=\ve f(v(\cdot))
  \end{equation}
 subject to 
 \begin{equation}\label{E:BCEx}
v(0)= 1+\ve g_1(v(t_1))\text{ and }v(\pi)=(-1)^n+\ve g_2(v(t_2)),
\ee
where $n\in \N$, $g_1$, $g_2$ are real-valued differentiable functions, and  $t_1$, $t_2\in [0,\pi]$.

In this specific case, solutions to the associated linear nonhomogeneous problem are given by 
\bes
 \bar{u}(x)=c\sin(nx)+\cos(nx).
\ees
A simple calculation shows that  \eqref{E:Imcond} and \eqref{E:Dercond} become
\be\label{ExIm}
n[g_1(\bar{u}(t_1))+(-1)^{n+1}g_2(\bar{u}(t_2))]=\int_0^\pi f(\bar{u}(s))\sin(ns)ds
\ee
and
\be\label{ExDer}
n[g'_1(\bar{u}(t_1))\sin(n\cdot t_1)+(-1)^{n+1}g_2'(\bar{u}(t_2))\sin(n\cdot t_2)]\neq\int_0^\pi f'(\bar{u}(s))\sin^2(ns)ds,
\ee
respectively.

If we take $f(x)=x^2$, then \eqref{ExIm} and \eqref{ExDer} become
\be\label{ExIm2}
n[g_1(\bar{u}(t_1))+(-1)^{n+1}g_2(\bar{u}(t_2))]=\dfrac{(2c^2+1)(1+(-1)^{n+1})}{3n}
\ee
and
{\small
\be\label{ExDer2}
n[g'_1(\bar{u}(t_1))\sin(n\cdot t_1)+(-1)^{n+1}g_2'(\bar{u}(t_2))\sin(n\cdot t_2)]\neq\dfrac{4c(1+(-1)^{n+1})}{3n},
\ee
}
respectively.
It is clear that there are an abundance of functions $g_1, g_2$ and points $t_1, t_2\in [0,\pi]$ such that \eqref{ExIm2} and \eqref{ExDer2} are satisfied. 

For a concrete example, take $g_1(x)=x^m$ for some $m\in \N, m>2,$ and assume $g_2(x)=(-1)^{n}K$, where $0<K$. In this case, \eqref{ExIm2} and \eqref{ExDer2} become
\be\label{ExIm3}
(c\sin(n\cdot t_1)+\cos(n\cdot t_1))^m-K=\dfrac{(2c^2+1)(1+(-1)^{n+1})}{3n^2}
\ee
and

\be\label{ExDer3}
m(c\sin(n\cdot t_1)+\cos(n\cdot t_1))^{m-1}\sin(n\cdot t_1)\neq \dfrac{4c(1+(-1)^{n+1})}{3n^2},
\ee
respectively.

Suppose now that $t_1\neq \dfrac{k\pi}{n}$ for each $k\in\{1, \cdots, n\}$.  It can be shown that there exists a $\bar{c}\in \R$ such that 
\be\label{ExIm4}
(\bar{c}\sin(n\cdot t_1)+\cos(n\cdot t_1))^m-K=\dfrac{(2\bar{c}^2+1)(1+(-1)^{n+1})}{3n^2}.
\ee
To see this, define $j:\R\to \R$ by
\bes
j(w)=(w\sin(n\cdot t_1)+\cos(n\cdot t_1))^m-K-\dfrac{(2w^2+1)(1+(-1)^{n+1})}{3n^2}.
\ees
 If we take $w=-\cot(n\cdot t_1)$, then  
\bes
j(w)=-\left(K+\dfrac{(2(-\cot(n\cdot t_1))^2+1)(1+(-1)^{n+1})}{3n^2}\right)<0;
\ees
depending on whether $m$  is odd or even, and also the sign of $\sin(n\cdot t_1)$, we have $\lim_{w\to -\infty}j(w)=\infty$  or $\lim_{w\to \infty}j(w)=\infty$.  The existence of a $\bar{c}$ such that \eqref{ExIm4} holds now follows from intermediate value theorem. 

If we now assume that $n\in 2\N$, then 
\eqref{ExDer3} becomes
\bes
m(c\sin(n\cdot t_1)+\cos(n\cdot t_1))^{m-1}\sin(n\cdot t_1)\neq 0.
\ees
However, $m(\bar{c}\sin(n\cdot t_1)+\cos(n\cdot t_1))^{m-1}\sin(n\cdot t_1)=0$  if and only if $\sin(n\cdot t_1)=0$, which is not the case, since $t_1\neq \dfrac{k\pi}{n}$ for each $k\in\{1, \cdots, n\}$. Thus, we conclude from Theorem \ref{T:mainresonance} the existence of solutions to 
 \begin{equation}\label{E:ODEEx}
v''+n^2v=\ve v^2
  \end{equation}
 subject to 
 \begin{equation}\label{E:BCEx}
v(0)= 1+\ve v^m(t_1)\text{ and } v(\pi)=1+\ve K
\ee
for ``small'' $\ve$ (small depends on $m, n$, $K$, and $t_1$), for all $m\in \N$ with $m>2$, any $n\in 2\N$,  every $K>0$, and any $t_1\neq \dfrac{k\pi}{n}$ for each $k\in\{1, \cdots, n\}$.
%



 \end{document}